%% file: main.tex
\documentclass{article}

\input{packages.tex}

\title{Equilibrium fluctuations for the Symmetric Exclusion Process on a compact Riemannian manifold}
\author{Bart van Ginkel\footnote{G.J.vanGinkel@tudelft.nl} \and Frank Redig\footnote{F.H.J.Redig@tudelft.nl}}
\date{%
    TU Delft\\
    \today
}

\begin{document}
\maketitle
\input{abstract.tex}
\input{introduction.tex}
\input{preliminaries.tex}
\input{martingales.tex}
\input{tightness.tex}
\input{uniqueness.tex}

\section*{Acknowledgement}
The authors thank Richard Kraaij for helpful discussions and for pointing out the reference~\citep{jakubowski1986skorokhod}.
The support of the grant 613.009.112 of the Netherlands Organisation for Scientific Research (NWO) is gratefully acknowledged.

\phantomsection
\bibliographystyle{abbrvnat}
\bibliography{refs}
\end{document}

%% file: packages.tex
\usepackage{amsmath,amsthm,mathrsfs}
\usepackage{bbm}
\usepackage{amsfonts}
\usepackage{anysize}
\usepackage{titlesec}
\usepackage{hyperref}
\usepackage[nottoc]{tocbibind}
\usepackage[titletoc]{appendix}
\usepackage[font=small,labelfont=bf]{caption}
\usepackage{tikz}
\usepackage{enumerate}

\usepackage[comma, sort&compress, authoryear]{natbib} 
\usepackage{hyperref}
\hypersetup{ colorlinks=true, linkcolor=blue, citecolor=blue,
  filecolor=blue, urlcolor=blue}

\newtheoremstyle{mystyle}
  {}
  {}
  {}
  {}
  {}
  {.}
  { }
  {{\underline{\thmname{#1}\thmnumber{ #2}~\thmnote{(#3)}}}}

\newtheoremstyle{mystyle2}
  {}
  {}
  {\itshape}
  {}
  {}
  {.}
  { }
  {{\underline{\thmname{#1}\thmnumber{ #2}~\thmnote{(#3)}}}}

\theoremstyle{definition}
\theoremstyle{mystyle}
\newtheorem{defi}{Definition}[section]
\newtheorem{rmk}[defi]{Remark}
\newtheorem{lemma}[defi]{Lemma}

\theoremstyle{plain}
\theoremstyle{mystyle2}
\newtheorem{thm}[defi]{Theorem}
\newtheorem{prop}[defi]{Proposition}
\newtheorem{cor}[defi]{Corollary}

\makeatletter

\newcommand{\R}{\mathbb{R}}
\newcommand{\N}{\mathbb{N}}

\newcommand{\Z}{\mathbb{Z}}
\newcommand{\p}{\mathbb{P}}
\newcommand{\E}{\mathbb{E}}
\newcommand{\var}{\mathrm{Var}}

\newcommand{\dd}{\mathrm{d}}

\newcommand{\1}{\mathbbm{1}}

\newcommand{\D}{\mathscr{D}}

\newcommand{\LL}{\mathcal{L}}
\newcommand{\el}{\mathscr{L}}

\newcommand{\Cov}{\mathrm{Cov}}

\newcommand{\grad}{\nabla}

%% file: abstract.tex
\begin{abstract}
We consider the Symmetric Exclusion Process on a compact Riemannian manifold, as introduced in~\cite{vanGinkel/Redig:2018}. There it was shown that the hydrodynamic limit satisfies the heat equation. In this paper we study the equilibrium fluctuations around this hydrodynamic limit. We define the fluctuation fields as functionals acting on smooth functions on the manifold and we show that they converge in distribution in the path space to a generalized Ornstein-Uhlenbeck process. This is done by proving tightness and by showing that the limiting fluctuations satisfy the corresponding martingale problem.
\end{abstract}

%% file: introduction.tex
\section{Introduction}

Over the past decades a lot of research has been done within the field of (mathematical) interacting particle systems to understand the emergence of macroscopic phenomena (like shock waves or spontaneous magnetization) and/or the behaviour of macroscopic observables (like pressure or total magnetization) from the dynamics of microscopic constituents (like particles or spins). The study of hydrodynamic limits concerns the derivation of the PDEs that govern macroscopic quantities from the (rescaled) dynamics of microscopic particle configurations. Hydrodynamic limits have been obtained for a large variety of interacting particle systems and by now there are multiple well-established methods to do this (see for instance~\cite{demasi2006mathematical}, \cite{kipnis1999scaling}). One can think of a hydrodynamic limit as a generalized law of large numbers. The natural question that one would like to answer after obtaining such result is how the limiting density field fluctuates around its (deterministic) hydrodynamic limit. In other words: one would like to find a corresponding (infinite dimensional) central limit theorem. For these so-called fluctuations a lot of models have been studied and by now standard methods have been established (see for instance~\cite[Chapter 11]{kipnis1999scaling}).

Most of the results that were described above are set in (a subset of) some Euclidean space. However, some phenomena are naturally modelled in a space that is not Euclidean. E.g., one could think of the motion of proteins along cell membranes. Apart from motivation from potential physics applications, it is also an important mathematical challenge to understand the influence on hydrodynamic limits and their fluctuations of geometric properties of the underlying space such as curvature. Therefore, it is worthwhile to extend the study of interacting particle systems to non-Euclidean spaces. 
Here one could think of spaces with a fractal structure, such as the Sierpinski gasket. In this area results about the Symmetric Exclusion Process have been obtained in for instance~\cite{Jara2009} and~\cite{chen2019phase}. A big advantage of the fractals such as the Sierpinski gasket is that there is a natural discretization available in the definition of the fractal.

In this paper, we are interested in this kind of results on Riemannian manifolds.
In~\cite{vanGinkel/Redig:2018} it was shown that one can set up the theory of hydrodynamic limits on a compact Riemannian manifold by defining suitable grid approximations of the manifold to define the microscopic particle systems. More precisely, we proved that the hydrodynamic limit of the Symmetric Exclusion Process on these grids is the heat equation on the manifold. Moreover, we showed that such grid approximations exist and can be obtained by sampling points uniformly from the manifold, and connecting them with edge weights depending on the Riemannian distance. 

In this paper, we continue the study of the exclusion process on a compact Riemannian manifolds by looking at the trajectory of the density fluctuation field. We consider the Symmetric Exclusion Process started from equilibrium (so from the product of Bernoulli measures with fixed intensity $\rho$) and we show that the corresponding fluctuation fields converge in law in the space of distribution-valued trajectories to a generalized Ornstein-Uhlenbeck process.
To do this, we follow the method that is described in~\cite[Chapter 11]{kipnis1999scaling}, i.e. we show tightness and we prove that any limiting distribution satisfies the same martingale problem.\\
Working on a manifold instead of $\R^d$ or the torus poses several new challenges. The first challenge is how to discretize the manifold in a suitable way. As we mentioned earlier, this was dealt with in~\cite{vanGinkel/Redig:2018}. The second challenge is to make sense of the fluctuation fields in the right space. The computations for tightness in a negatively indexed Sobolev space as performed in~\cite{kipnis1999scaling} become intrinsically more involved on a manifold due to the absence of notions like translation and translation invariance, which implies that one cannot rely on standard Fourier analysis for discrete and continuous Laplacians. Therefore we resort to define the fluctuations fields as elements of the dual of the smooth functions $C^\infty(M)$ on the manifold. The advantage of this approach is that $C^\infty(M)$ is a nuclear space, which ensures that we only need to prove tightness of the distribution-valued trajectories applied to test functions. The drawback, however, is that the dual of $C^\infty(M)$ does not have a norm and that the space of cadlag trajectories is not even metrizable. Therefore we must be careful when treating the convergence of the martingales in section~\ref{sec:uniqueness}.\\
Finally, note that the result of this paper is also a constructive proof for the existence of generalized Ornstein-Uhlenbeck processes on compact Riemannian manifolds. This type of processes have been studied (through their corresponding SPDEs) in for instance~\cite{christensen1985linear}. In our work, we show that the fluctuations fields converge to a well-defined limiting random field and that this field satisfies the martingale problem that is associated to a generalized Ornstein-Uhlenbeck process.

\subsection*{Overview of the paper}
In section~\ref{sec:prelim} we define the grids that approximate the manifold, the Symmetric Exclusion Process on the grids and the corresponding fluctuations fields as random elements of $D([0,T],(C^\infty)')$. We also state the main theorem and give a brief overview of the proof. Then in section~\ref{sec:mgales} we study the Dynkin martingale associated to the fluctuation fields. In section~\ref{sec:tightness} we prove tightness of the distributions on $D([0,T],(C^\infty)')$ of the fluctuation fields. Finally, in section~\ref{sec:uniqueness}, we show that all possible limiting measures of subsequences are the same, by showing that they satisfy the same martingale problem with the same initial conditions.

%% file: preliminaries.tex
\section{Preliminaries}\label{sec:prelim}
In all of this paper we fix a compact Riemannian manifold $M$. In this section we will introduce approximating grids on $M$ and the Symmetric Exclusion Process on these grids. Further we state the theorem and give an outline of its proof.

\subsection*{Definitions}
In order to define interacting particle systems on a manifold, we need a suitable discretization of the manifold. In $\R^d$ such disretization is easily obtained by taking $\frac{1}{N} \Z^d$ (or something similar). A manifold, however, does not have these nice scaling properties, so another path must be taken (this is explained further in~\cite{vanGinkel/Redig:2018}). 

Let $(G^N,c^N)_{N=1}^\infty$ be a sequence of grids with $G^N=\{p^N_1,..,p^N_N\}\subset M$ (we usually write simply $p_i$ instead of $p_i^N$) and edge weights $c^N=\{c^N_{ij}\}_{i,j\leq N}$ where $c^N_{ij}$ is the weight of the edge between $p^N_i$ and $p^N_j$ and we assume that $c^N_{ij}=c^N_{ji}\geq 0$ for all $i,j\leq N$. We denote by $\el^N$ the corresponding graph Laplacians
\begin{equation*}
    \el^Nf(p_i) = \sum_{j=1}^N c_{ij}^N (f(p_j)-f(p_i)).
\end{equation*}
Note that this operator (acting on functions $G^N\rightarrow \R$) generates a random walk on $G_N$ with jumping rates $c_N$.
We assume that the graph Laplacians converge to the Laplace-Beltrami operator $\Delta_M$ in a uniform way, i.e. for all $f\in C^\infty$
\begin{equation}\label{approxgrid}
    \lim_{N\rightarrow\infty} \sup_{1\leq i\leq N} \left| \sum_{j=1}^N c_{ij}^N (f(p_j)-f(p_i))-\Delta_Mf(p_i)\right|=0.
\end{equation}
It will be convenient later to have the following notation for fixed smooth $f$
\begin{equation*}
    E_f(N)=\sup_{i\leq N} \left| \el^Nf(p_i)-\Delta_Mf(p_i)\right|.
\end{equation*}
Note that by assumption~\eqref{approxgrid} for each smooth $f$, $E_f(N)$ goes to $0$ as $N$ goes to infinity.

Finally we assume that the empirical measures corresponding to the $G^N$ converge weakly to the normalized volume measure $\overline V$ on $M$, i.e. for all continuous $f$
\begin{equation*}
    \frac{1}{N}\sum_{i=1}^N f(p_i) = \int f \dd \left(\frac{1}{N}\sum_{i=1}^N \delta_{p_i}\right) \rightarrow \int f \dd \overline V \hspace{1cm} (N\rightarrow \infty).
\end{equation*}
\begin{rmk}
Note that usually for results about hydrodynamic limits or fluctuations there is an explicit time scaling visible in the equations, i.e. in the diffusive case (like for the Symmetric Exclusion Process) one would typically consider $N^2 \el^N$. However, in our case this rescaling is hidden in the conductances $c^N_{ij}$ because of the assumption in~\eqref{approxgrid}. The reason for this approach is that it is less straightforward to define the space scale in a more general grid than a lattice. To see how the (diffusive) space and time scales do show up in a particular construction of such grid, see~\cite[Remark 3.5]{vanGinkel/Redig:2018}.
\end{rmk}
\begin{rmk}
We formulate the results of this paper in terms of these general uniformly approximating grids. Note, however, that these grids can always be obtained. It was shown in~\cite{vanGinkel/Redig:2018} that sampling a sequence of iid random uniformly random elements from the manifold and setting $G^N$ to be the first $N$ of these elements yields a suitable sequence of grids with probability $1$.
\end{rmk}
We can now define the Symmetric Exclusion Process (SEP) on $G^N$. The idea of this process is that it describes particles that perform independent random walks according to the jumping rates $c^N$ with the restriction that jumps to occupied sites are suppressed. Note that in this way every site contains at most one particle. Therefore the SEP $\eta^{N}=(\eta^{N}_t)_{t\geq0}$ takes values in $\{0,1\}^{G^N}$, where $1$ denotes the presence of a particle and $0$ the absence. An equivalent way to describe the process is by saying that the edges have clocks that ring according to the rates $c^N$ and that if particles are present at either ends of the edge, they jump to the other end. Therefore the dynamics are defined through the generator
\begin{equation}
    L^{N}h(\eta)=\frac{1}{2}\sum_{i,j=1}^Nc_{ij}^N (h(\eta^{ij})-h(\eta)), \quad h: \{0,1\}^{G^N}\rightarrow\R,\label{SEPgen}
\end{equation}
where $\eta^{ij}:=\eta^{p_ip_j}$ denotes the configuration obtained from $\eta$ by exchanging the values at $p_i$ and $p_j$. 

Since we want to consider equilibrium fluctuations, we want to initialize the process in a stationary measure. Therefore fix $\rho\in(0,1)$ and set as the initial configuration $\nu^N_\rho$: the product of $N$ Bernoulli distributions with parameter $\rho$. It is well kwown that the SEP is reversible with respect to this measure, so in particular that this measure is invariant for the SEP.

Now we define the fluctuation field $Y^N\in D([0,T],(C^\infty)')$ through the following action on smooth functions $f\in C^\infty$
\begin{equation}
    Y^N_t(f) = \frac{1}{\sqrt{N}} \sum_{i=1}^N f(p_i)(\eta^N_{t}(p_i)-\rho).\label{def:field}
\end{equation}
The law of the underlying process $\eta$ induces a law $\LL_N$ of the density fields on $D([0,T],(C^\infty)')$.

\begin{rmk}
To see that~\eqref{def:field} is the right object with the right scaling, note the following. First of all at fixed times $\eta^N_t$ is distributed like a product of Bernoulli measures, so it is a very rough object and it makes sense to regard it as acting on functions (instead of considering its pointwise values). Second, the expectation of $\eta_t^N(p_i)$ equals $\rho$ for every grid point, so the right quantity is subtracted. This makes sure that for any $f$, $\E Y_t^N(f)=0$. Finally, $\var Y_t^N(f)$ equals
\begin{equation*}
    \var\left(\frac{1}{\sqrt{N}} \sum_{i=1}^N f(p_i)(\eta^N_{t}(p_i)-\rho)\right) = \frac{1}{N} \sum_{i=1}^Nf^2(p_i)\rho(1-\rho) \rightarrow \rho(1-\rho) \int f^2\dd \overline V,
\end{equation*}
where we use that $f$ is continuous and that the empirical measure of the grid points converges to the uniform measure on the manifold. 
This motivates that $1/\sqrt{N}$ provides the right scaling to get a meaningful, non-degenerate limit.
\end{rmk}

It is natural to expect that the fluctuation field converges to a generalized stationary Ornstein-Uhlenbeck process. This process is the solution of the following (formal) SPDE
\begin{equation}\label{eq:OU-SPDE}
    \dd Y_t = \frac{1}{2} \Delta Y_t \dd t + \sqrt{\rho(1-\rho)} \nabla \dd W_t,
\end{equation}
where $Y$ takes values in $D([0,T],(C^\infty)')$ and $W_t$ is space-time white noise. A process $Y$ is a mild solution of~\eqref{eq:OU-SPDE} if for any $f\in C^\infty$
\begin{equation*}
    Y_t(f)= Y_0(f) + \sqrt{\rho(1-\rho)}\int_0^t \nabla S_{t-s}f \dd W_s,
\end{equation*}
where $(S_t,t\geq 0)$ is the semigroup corresponding to Brownian motion. The solution is a Gaussian process that is stationary with respect to white noise $W_0$ with covariance
\begin{equation*}
    \Cov(W_0(f),W_0(g))=\rho(1-\rho)\langle f,g\rangle
\end{equation*} 
and with stationary covariance
\begin{equation}\label{eq:covOU}
    \Cov(Y_t(f),Y_s(g))=\rho(1-\rho)\langle f,S_{|t-s|} g\rangle.
\end{equation}
More precisely, this Ornstein-Uhlenbeck process is defined via the following martingale problem. For each test function $f$ the following are martingales:
\begin{eqnarray}
    M_t^f&:=&Y_t(f)-Y_0(f) - \int_0^t Y_s(\Delta_Mf)\dd s \label{eq:OU-mgale-def}\\
    N_t^f&:=& (M_t^f)^2 - 2t\rho(1-\rho)\int (\grad f)^2 \dd \overline V.\nonumber
\end{eqnarray}

\subsection*{Main theorem and overview of the proof}
The main theorem of this paper is the following.
\begin{thm}\label{main theorem}
There exists a random element $Y$ of $C([0,T],(C^\infty)')\subset D([0,T],(C^\infty)')$ with corresponding law $\LL$ on $D([0,T],(C^\infty)')$ such that $\LL_N\rightarrow \LL$ as $N$ goes to infinity. Moreover, this $Y$ is a generalized Ornstein-Uhlenbeck process solving the martingale problem~\eqref{eq:OU-mgale-def}.
\end{thm}
In other words, the theorem says that as $N$ approaches infinity the trajectories of fluctuations converge to a generalized Ornstein-Uhlenbeck process.

The proof consists of two parts. In section~\ref{sec:tightness} we will show tightness of $(\LL_N,N\in\N)$. By the first part of the proof of proposition 5.1 from~\cite{mitoma1983tightness}, this implies that every subsequence of $(\LL_N,N\in\N)$ has a further subsequence that converges to some limit. Then in section~\ref{sec:uniqueness} we show that all limiting points are the same. This is done by showing that any limiting measure satisfies the same martingale problem with the same initial condition. This martingale problem also characterizes the limiting process as a generalized Ornstein-Uhlenbeck process like described above and we compute the limiting covariance to confirm this. Together these results imply theorem~\ref{main theorem}. To do all this, we start by analyzing martingales involving the fluctuation fields in section~\ref{sec:mgales}.

%% file: martingales.tex
\section{Dynkin martingale}\label{sec:mgales}
We know that $\eta^N$ is a Markov process with generator given by~\eqref{SEPgen} (we will usually leave out the superscript $N$). Now if we fix $f\in C^\infty$, we can define the function
\begin{equation*}
    \phi^{N,f}=\{0,1\}^{G^N}\longrightarrow \R,\quad \eta\mapsto \frac{1}{\sqrt N}\sum_{i=1}^Nf(p_i)(\eta(p_i)-\rho).
\end{equation*}
Now we know that both the Dynkin martingale
\begin{equation*}
    M^{N,f}_t:=\phi^{N,f}(\eta_t)-\phi^{N,f}(\eta_0)-\int_0^t L^N\phi^{N,f}(\eta_s)\dd s
\end{equation*}
and 
\begin{equation*}
    N^{N,f}_t:=\left(M^{N,f}_t\right)^2-\int_0^t L^N(\phi^{N,f})^2(\eta_s)-2\phi^{N,f}(\eta_s)L^N\phi^{N,f}(\eta_s)\dd s
\end{equation*}
are martingales with respect to the natural filtration generated by $\eta$ (for this well-known result and approach see for instance~\cite{seppalainen2008translation} or \cite{kipnis1999scaling}). Since these martingales will have an important role in the calculations later, we will calculate the components that are involved and study their limits. 

\subsection*{The martingale $M^{N,f}$}
First of all note that
\begin{equation*}
    \phi^{N,f}(\eta_t) = Y^N_t(f).
\end{equation*}
Second we want to calculate $L^N\phi^{N,f}(\eta)$ for $\eta\in \{0,1\}^{G^N}$. To do this, first we see
\begin{eqnarray*}
    \phi^{N,f}(\eta^{ij})-\phi^{N,f}(\eta) &=& \frac{1}{\sqrt N}\sum_{k=1}^N f(p_k)(\eta^{ij}(p_k)-\rho) - \frac{1}{\sqrt N}\sum_{k=1}^N f(p_k)(\eta(p_k)-\rho) \\
    &=& \frac{1}{\sqrt N}\sum_{k=1}^N f(p_k)(\eta^{ij}(p_k)-\eta(p_k)) = \frac{1}{\sqrt N} (f(p_i)(\eta(p_j)-\eta(p_i))+f(p_j)(\eta(p_i)-\eta(p_j)))\\
    &=& \frac{1}{\sqrt N} (\eta(p_i)(f(p_j)-f(p_i))+\eta(p_j)(f(p_i)-f(p_j))).
\end{eqnarray*}
Now we obtain that
\begin{eqnarray*}
    L^N\phi^{N,f}(\eta)=\frac{1}{2} \sum_{i,j=1}^N c_{ij}^N \frac{1}{\sqrt N} (\eta(p_i)(f(p_j)-f(p_i))+\eta(p_j)(f(p_i)-f(p_j))).
\end{eqnarray*}
By symmetry of the weights, this equals
\begin{eqnarray}
    \sum_{i,j=1}^N c_{ij}^N \frac{1}{\sqrt N} \eta(p_i)(f(p_j)-f(p_i)) &=& 
    \frac{1}{\sqrt N} \sum_{i=1}^N \eta(p_i) \sum_{j=1}^N c_{ij}^N (f(p_j)-f(p_i))\nonumber \\
    &=& \frac{1}{\sqrt N} \sum_{i=1}^N \eta(p_i) \el^Nf(p_i),\label{eq1}
\end{eqnarray}
where we recall that $\el^N$ is the generator of the random walk according to the weights $c_{ij}^N$ on $G^N$. Now note that
\begin{equation*}
    \frac{1}{\sqrt N} \sum_{i=1}^N \rho \el^Nf(p_i) = \frac{\rho}{\sqrt{N}} \sum_{i,j=1}^N c_{ij}^N(f(p_j)-f(p_i)) = 0
\end{equation*}
to see that~\eqref{eq1} equals
\begin{equation*}
    \frac{1}{\sqrt N} \sum_{i=1}^N  \el^Nf(p_i)(\eta(p_i)-\rho).
\end{equation*}
This implies that 
\begin{equation}\label{Dmgale}
    M_t^{N,f}= Y_t^N(f)-Y_0^N(f)-\int_0^t Y_s^N(\el^N f)\dd s.
\end{equation}
The next lemma shows that as $N$ grows to infinity we can replace $\el^N$ by $\Delta_M$.
\begin{lemma}\label{lemma:replacelaplacian}
For all $f\in C^\infty$,
\begin{equation*}
    \lim_{N\rightarrow\infty} \E \left(\int_0^t Y_s^N(\el^N f)\dd s - \int_0^t Y_s^N(\Delta_M f)\dd s\right)^2 = 0.
\end{equation*}
\end{lemma}
\begin{proof}
First we see
\begin{eqnarray}
    &&\E \left(\int_0^t Y_s^N(\el^N f)\dd s - \int_0^t Y_s^N(\Delta_M f)\dd s\right)^2 =
\E \left( \int_0^t Y^N_s(\Delta_Mf-\el^Nf)\dd s \right)^2\nonumber\\
    &\leq& t \int_0^t \E Y^N_s(\Delta_Mf-\el^Nf)^2 \dd s.\label{eq:lemma3.1}
\end{eqnarray}
Now, using that $\eta_t$ is a vector of independent Bernoulli random variables, we compute
\begin{eqnarray*}
    \E Y^N_s(\Delta_Mf-\el^Nf)^2 = \frac{1}{N} \sum_{i=1}^N\left(\Delta_Mf(p_i)-\el^Nf(p_i)\right)^2 \rho(1-\rho) \leq \rho(1-\rho)E_f(N)^2,
\end{eqnarray*}
so~\eqref{eq:lemma3.1} is bounded by $t^2 \rho(1-\rho)E_f(N)^2$, which vanishes in the limit.
\end{proof}

\subsection*{The martingale $N^{N,f}$}
Now we analyze the second martingale. First we calculate the integrand, which we will denote by $\Gamma^{N,f}(s)$, and its expectation and variance.
\begin{lemma}
For all $f\in C^\infty$,
\begin{equation*}
    \lim_{N\rightarrow\infty} \E \Gamma^{N,f}(s) = 2\rho(1-\rho)\int (\nabla f)^2 \dd \overline V.
\end{equation*}
\end{lemma}
\begin{proof}
First we calculate
\begin{eqnarray}
   \Gamma^{N,f}(s) &=& L^N(\phi^{N,f})^2(\eta_s)-2\phi^{N,f}(\eta_s)L^N\phi^{N,f}(\eta_s) \label{carchaexp}\\
    &=& \frac{1}{2} \sum_{i,j=1}^N c_{ij}^N\left(\phi^{N,f}(\eta_s^{ij})-\phi^{N,f}(\eta_s)\right)^2\nonumber\\
    &=& \frac{1}{2} \sum_{i,j=1}^N c_{ij}^N \left(\frac{1}{\sqrt N} \bigg[\eta_s(p_i)(f(p_j)-f(p_i))+\eta_s(p_j)(f(p_i)-f(p_j))\bigg]\right)^2\nonumber\\
    &=& \frac{1}{2N} \sum_{i,j=1}^N c_{ij}^N (\eta_s(p_j)-\eta_s(p_i))^2(f(p_j)-f(p_i))^2\label{quadvar}\\
    &=& -\frac{1}{N} \sum_{i,j=1}^N c_{ij}^N (f(p_j)-f(p_i))f(p_i) (\eta_s(p_j)-\eta_s(p_i))^2\nonumber.
\end{eqnarray}
Now we take the expectation and, using that $\E (\eta_s(p_j)-\eta_s(p_i))^2=2\rho(1-\rho)$, we obtain
\begin{eqnarray}
    \E \Gamma^{N,f}(s)&=& -\frac{1}{N} \sum_{i,j=1}^N c_{ij}^N (f(p_j)-f(p_i))f(p_i) \E (\eta_s(p_j)-\eta_s(p_i))^2 \nonumber\\
    &=& -\frac{2\rho(1-\rho)}{N}\sum_{i=1}^Nf(p_i) \sum_{j=1}^N c_{ij}^N (f(p_j)-f(p_i)) \nonumber\\
    &=& -\frac{2\rho(1-\rho)}{N}\sum_{i=1}^Nf(p_i) \el^Nf(p_i).\label{expofquadvar}
\end{eqnarray}
Note that 
\begin{eqnarray*}
    \left| \frac{1}{N}\sum_{i=1}^Nf(p_i) \el^Nf(p_i) - \frac{1}{N}\sum_{i=1}^Nf(p_i) \Delta_M f(p_i)\right| &\leq& \frac{1}{N}\sum_{i=1}^N|f(p_i)|\left| \el^Nf(p_i) - \Delta_M f(p_i)\right|\\
    &\leq& \|f\|_\infty E_f(N) \longrightarrow 0 \hspace{1cm} (N\rightarrow \infty).
\end{eqnarray*}
This implies that 
\begin{equation}\label{quadvarconv}
    \lim_{N\rightarrow\infty}\frac{1}{N}\sum_{i=1}^Nf(p_i) \el^Nf(p_i) = \lim_{N\rightarrow\infty}\frac{1}{N}\sum_{i=1}^Nf(p_i) \Delta_M f(p_i) = \int f\Delta_M f\dd \overline V.
\end{equation}
Combining this with~\eqref{expofquadvar}, we conclude that 
\begin{equation}\label{convofexpofcarcha}
    \lim_{N\rightarrow\infty} \E \Gamma^{N,f}(s) = -2\rho(1-\rho) \int f\Delta_M f\dd \overline V = 2\rho(1-\rho) \int (\grad f)^2\dd \overline V.
\end{equation}
\end{proof}

Next we want to prove that $\var{(\Gamma^{N,f}(s))}$ vanishes in the limit, but we first need the following lemma and corollary.
\begin{lemma}\label{carchalemma}
For all $f\in C^\infty$,
\begin{equation}\label{carcha}
    \lim_{N\rightarrow\infty} \sup_{1\leq i\leq N} \left| \sum_{j=1}^N c_{ij}^N (f(p_j)-f(p_i))^2-\left(\Delta_M(f^2)(p_i)-2 f(p_i)\Delta_M f(p_i)\right)\right|=0.
\end{equation}
\end{lemma}
\begin{proof}
By writing $(f(p_j)-f(p_i))^2 = f(p_j)^2-f(p_i)^2 -2 f(p_i)(f(p_j)-f(p_i))$ and the triangle inequality we see that~\eqref{carcha} is bounded by
\begin{equation*}
    \sup_{1\leq i\leq N} \left| \sum_{j=1}^N c_{ij}^N (f(p_j)^2-f(p_i)^2)-\Delta_M(f^2)(p_i)\right| + 2 \|f\|_\infty \sup_{1\leq i\leq N} \left| \sum_{j=1}^N c_{ij}^N (f(p_j)-f(p_i))-\Delta_f(p_i)\right|,
\end{equation*}
which goes to $0$ by~\eqref{approxgrid}.
\end{proof}
\begin{cor}\label{boundofcarcha}
For all $f\in C^\infty$, there exists a constant $C$ (depending on $f$) such that for all $N\in \N$
\begin{equation*}
    \sup_{1\leq i\leq N} \sum_{j=1}^N c_{ij}^N (f(p_j)-f(p_i))^2 \leq C.
\end{equation*}
\end{cor}
\begin{proof}
For all $i\leq N$ by lemma~\ref{carchalemma}
\begin{eqnarray*}
    \sum_{j=1}^N c_{ij}^N (f(p_j)-f(p_i))^2 &\leq& \sup_{1\leq i\leq N} \left|\Delta_M(f^2)(p_i)-2 f(p_i)\Delta_f(p_i)\right| \\
    &+& \sup_{1\leq i\leq N} \left| \sum_{j=1}^N c_{ij}^N (f(p_j)-f(p_i))^2-\left(\Delta_M(f^2)(p_i)-2 f(p_i)\Delta_f(p_i)\right)\right| \\
    &\leq& \|\Delta_M f^2\|_\infty + 2 \|f\|_\infty\|\Delta_Mf\|_\infty + h_f(N),
\end{eqnarray*}
where $h_f(N)=o(1)$. Since this bound does not depend on $i$ and is bounded in $N$, the result follows.
\end{proof}
Now we can prove the following lemma.
\begin{lemma}
For all $f\in C^\infty$,
\begin{equation*}
    \lim_{N\rightarrow\infty} \var \Gamma^{N,f}(s) = 0.
\end{equation*}
\end{lemma}
\begin{proof}
Using~\eqref{quadvar}, we see that the variance of~\eqref{carchaexp} equals
\begin{equation*}
    \frac{1}{4N^2} \sum_{i,j,k,l=1}^N c_{ij}^Nc_{kl}^N(f(p_j)-f(p_i))^2(f(p_l)-f(p_k))^2 \Cov((\eta_s(p_j)-\eta_s(p_i))^2,(\eta_s(p_l)-\eta_s(p_k))^2).
\end{equation*}
Now note that $|\Cov((\eta_s(p_j)-\eta_s(p_i))^2,(\eta_s(p_l)-\eta_s(p_k))^2)|\leq 1$ since both random variables in the covariance take values in $\{0,1\}$. Moreover,
\begin{eqnarray*}
    \Cov((\eta_s(p_j)-\eta_s(p_i))^2,(\eta_s(p_l)-\eta_s(p_k))^2)\neq 0 &\implies& (\eta_s(p_j)-\eta_s(p_i))^2\not\perp (\eta_s(p_l)-\eta_s(p_k))^2 \\
    &\implies& (i=k \text{ or } i=l \text{ or } j=k \text{ or } j=l).
\end{eqnarray*}
Together, this implies that
\begin{equation*}
    |\Cov((\eta_s(p_j)-\eta_s(p_i))^2,(\eta_s(p_l)-\eta_s(p_k))^2)|\leq \delta_{ik}+\delta_{il}+\delta_{jk}+\delta_{jl}.
\end{equation*}
Now it by positivity of the summands and symmetry it suffices to show that
\begin{equation}\label{carchawithind}
    \frac{1}{4N^2} \sum_{i,j,k,l=1}^N c_{ij}^Nc_{kl}^N(f(p_j)-f(p_i))^2(f(p_l)-f(p_k))^2 \delta_{ik}
\end{equation}
goes to $0$. By rearranging we see that~\eqref{carchawithind} equals
\begin{equation}\label{carchawithoutind}
    \frac{1}{4N^2} \sum_{i,j=1}^N c_{ij}^N(f(p_j)-f(p_i))^2\sum_{l=1}^Nc_{il}^N(f(p_l)-f(p_i))^2.
\end{equation}
By corollary~\ref{boundofcarcha}, there exists $C=C(f)>0$ such that~\eqref{carchawithoutind} is bounded by
\begin{eqnarray*}
    &&\frac{C}{4N^2} \sum_{i,j=1}^N c_{ij}^N(f(p_j)-f(p_i))^2\\
    &=& -\frac{C}{2N} \frac{1}{N} \sum_{i=1}^N f(p_i) \sum_{j=1}^N c_{ij}^N (f(p_j)-f(p_i))\\
    &\longrightarrow& 0\cdot \int f\Delta_M f\dd \overline V = 0,
\end{eqnarray*}
where in the last line we used~\eqref{expofquadvar} and~\eqref{quadvarconv}. This implies that 
\begin{equation}\label{convofvarofcarcha}
    \lim_{N\rightarrow\infty} \var\left(\Gamma^{N,f}(s)\right) = 0.
\end{equation}
\end{proof}

Putting this together, we can prove the following.
\begin{lemma}\label{lemma:quadvarconv}
For all $f\in C^\infty$,
\begin{equation*}
    \lim_{N\rightarrow\infty} \E \left(\int_0^t \Gamma^{N,f}(s)\dd s - 2t\rho(1-\rho)\int (\nabla f)^2 \dd \overline V\right)^2 = 0.
\end{equation*}
\end{lemma}
\begin{proof}
Combining~\eqref{convofexpofcarcha} and~\eqref{convofvarofcarcha} and realizing that in both cases the speed of convergence does not depend on $s$, we obtain
\begin{eqnarray*}
    &&\E \left(\int_0^t \Gamma^{N,f}(s)\dd s - 2t\rho(1-\rho)\int (\nabla f)^2 \dd \overline V\right)^2 \leq t \int_0^t\E \left( \Gamma^{N,f}(s) - 2\rho(1-\rho)\int (\nabla f)^2 \dd \overline V\right)^2 \dd s\\
    &=& t^2 \left( \var\Gamma^{N,f}(0) + \left(\E\Gamma^{N,f}(0) - 2\rho(1-\rho)\int (\nabla f)^2 \dd \overline V\right)^2\right)\rightarrow 0.
\end{eqnarray*}
We used here that for a constant $c$, $\E (X-c)^2 = \var(X)+(\E X-c)^2$.
\end{proof}

%% file: tightness.tex
\section{Tightness}\label{sec:tightness}
In this section we show tightness. Note that for fixed $N\in\N$ and $f\in C^\infty$, $Y^N_\cdot(f)$ is a trajectory in $D([0,T],\R)$. 
Since $C^\infty$ is a nuclear space, by~\cite[Thm 4.1]{mitoma1983tightness} it suffices to prove for fixed $f\in C^\infty$ that $(Y^N_\cdot(f),N\in\N)$ a tight collections of random elements of $D([0,T],\R)$. By~\cite[Section 4.1]{kipnis1999scaling} it suffices to show that Aldous' criterion holds, i.e. that 
\begin{enumerate}[(i)]
    \item for each $t\in [0,T]$ and $\epsilon>0$, there is a compact $K(t,\epsilon)\subset \R$ such that $$\sup_N \LL^N(Y_t^N(f)\notin K(t,\epsilon))\leq \epsilon$$\label{tight1}
    \item \label{tight2} for all $\epsilon >0$ $$\lim_{\gamma\rightarrow 0} \limsup_{N\rightarrow\infty} \sup_{\tau\in \mathscr{T}_T,\theta\leq \gamma} \LL^N\left(\left|Y_\tau^N(f)-Y_{\tau+\theta}^N(f)\right|>\epsilon\right)=0$$
\end{enumerate}
\begin{proof}
We study more carefully the distribution of $Y^N_t(f)$. Since $\nu^N_\rho$ is invariant under the SEP dynamics, we know that the $\eta_t(p_i)$'s are iid Bernoulli with parameter $\rho$. This means that $(\eta_t(p_i)-\rho)$ has mean $0$ and variance $\rho(1-\rho)$. Since they are independent, we see that $Y_t^N(f)$ has mean $0$ and variance
\begin{equation*}
    \var\left(\frac{1}{\sqrt N}\sum_{i=1}^N f(p_i) (\eta_t(p_i)-\rho)\right) = \frac{1}{N}\sum_{i=1}^N f^2(p_i)\rho(1-\rho).
\end{equation*}
By the central limit theorem, the distribution of $Y_t^N(f)$ converges to the $N\left(0,\rho(1-\rho)\int f^2\dd \overline V\right)$ distribution. This implies tightness of $(Y_t^N(f),N\in\N)$, which is~\eqref{tight1}.
For~\eqref{tight2} we use the Dynkin martingale representation~\eqref{Dmgale} to write
\begin{equation}
    Y_t^N(f)= M_t^{N,f}+Y_0^N(f)+\int_0^t Y_s^N(\el^N f)\dd s.\label{dynk}
\end{equation}
It suffices to show~\eqref{tight2} for the integral term and the martingale term of~\eqref{dynk}. For the integral term we first calculate the following (where $\E$ refers to the expectation with respect to $\LL^N$).
\begin{eqnarray}
    \E\left(\int_\tau^{\tau+\theta}Y^N_s(\el^Nf)\dd s\right)^2 &\leq&
    \theta\E  \int_\tau^{\tau+\theta} \left(Y^N_s(\el^Nf)\right)^2 \dd s \leq
    \theta\E  \int_0^T \left(Y^N_s(\el^Nf)\right)^2 \dd s \nonumber \\
    &=& \theta \int_0^T \E \left(Y^N_s(\el^Nf)\right)^2 \dd s
    = \theta \int_0^T \frac{1}{N}\sum_{i=1}^{N}\el^Nf(p_i)^2 \rho(1-\rho) \dd s \nonumber  \\
    &=& \theta T \rho(1-\rho) \frac{1}{N}\sum_{i=1}^{N}\el^Nf(p_i)^2 \label{expr}
\end{eqnarray}
Now note that
\begin{eqnarray*}
    \frac{1}{N}\sum_{i=1}^{N}\el^Nf(p_i)^2 &\leq& \frac{1}{N}\sum_{i=1}^{N}(|\Delta_Mf(p_i)|+E_{p_i}(N))^2
    \leq \frac{2}{N}\sum_{i=1}^{N}\Delta_Mf(p_i)^2 + \frac{2}{N}\sum_{i=1}^{N}E_{p_i}(N))^2\\
    &\leq& \frac{2}{N}\sum_{i=1}^{N}\Delta_Mf(p_i)^2 + 2 E(N)^2 \rightarrow 2\int (\Delta_Mf)^2\dd \overline V + 0,
\end{eqnarray*}
which implies that there exists some $C>0$ independent of $N$ such that~\eqref{expr} is bounded by $\theta T \rho(1-\rho)C$. Now we see that
\begin{eqnarray*}
    &&\lim_{\gamma\rightarrow 0} \limsup_{N\rightarrow\infty} \sup_{\tau\in \mathscr{T}_T,\theta\leq \gamma} \LL^N \left(\left|Y_\tau^N(f)-Y_{\tau+\theta}^N(f)\right|>\epsilon\right) \leq \lim_{\gamma\rightarrow 0} \limsup_{N\rightarrow\infty} \sup_{\tau\in \mathscr{T}_T,\theta\leq \gamma} \frac{1}{\epsilon^2}\E\left(\int_\tau^{\tau+\theta}Y^N_s(\el^Nf)\dd s\right)^2
    \\ &\leq& \lim_{\gamma\rightarrow 0} \limsup_{N\rightarrow\infty} \sup_{\tau\in \mathscr{T}_T,\theta\leq \gamma} \frac{\theta T \rho(1-\rho)C}{\epsilon^2} = \lim_{\gamma\rightarrow 0} \limsup_{N\rightarrow\infty} \frac{\gamma T \rho(1-\rho)C}{\epsilon^2} = 0.
\end{eqnarray*}
Now for the martingale term, by the martingale property we see that
\begin{eqnarray*}
    \E \left(M_{\tau+\theta}^{N,f}-M_{\tau}^{N,f}\right)^2 &=& \E \left(\langle M^{N,f},M^{N,f}\rangle_{\tau+\theta}-\langle M^{N,f},M^{N,f}\rangle_s\right)\\
    &=& \E \int_\tau^{\tau+\theta} L^N(\phi^{N,f})^2(\eta_s)-2\phi^{N,f}(\eta_s)L^N\phi^{N,f}(\eta_s)\dd s.
\end{eqnarray*}
By~\eqref{quadvar}, the latter equals
\begin{eqnarray*}
    &&\frac{1}{2N}\sum_{i,j=1}^{N}c_{ij}^N (f(p_j)-f(p_i))^2 \hspace{1mm} \E \int_\tau^{\tau+\theta} (\eta_s(p_i)-\eta_s(p_j))^2 \dd s \leq  \frac{1}{2N}\sum_{i,j=1}^{N}c_{ij}^N (f(p_j)-f(p_i))^2 \theta\\
    &=& - \theta \frac{1}{N}\sum_{i,j=1}^{N}c_{ij}^N (f(p_j)-f(p_i))f(p_i) = -\theta \frac{1}{N}\sum_{i=1}^N f(p_i) \el^Nf(p_i) \longrightarrow -\theta \int f\Delta_Mf\dd \overline V = \theta \int (\grad f)^2\dd \overline{V},
\end{eqnarray*}
where in the last line we used~\eqref{quadvarconv}. This implies that there exists some $C>0$ independent of $N$ such that
\begin{equation*}
    \LL^N\left(\left|M_{\tau+\theta}^{N,f}-M_{\tau}^{N,f}\right|>\epsilon\right)\leq \frac{1}{\epsilon^2}\E \left(M_{\tau+\theta}^{N,f}-M_{\tau}^{N,f}\right)^2  \leq \frac{\theta C}{\epsilon^2}.
\end{equation*}
As with the integral term, this implies~\eqref{tight2}.
\end{proof}

%% file: uniqueness.tex
\section{Uniqueness of limits of subsequences}\label{sec:uniqueness}

Now let $\LL^*$ be the limit of a subsequence $\LL_{N_k}$. We want to show that $\LL^*$ satisfies certain initial conditions and a martingale problem, which will then uniquely determine it.

First of all the initial condition can be shown to be a Gaussian field in exactly the same way as~\cite[Chapter 11 Lemma 2.1]{kipnis1999scaling}, i.e. $\LL^*$ restricted to $\mathscr{F}_0$ is a Gaussian field with covariance
\begin{equation}\label{initdist}
    \E[Y_0(f)Y_0(g)]=\rho(1-\rho)\int fg \dd \overline V.
\end{equation}

We will need the following lemma.
\begin{lemma}\label{contpaths}
Let $Y$ have distribution $\LL^*$. Then for each smooth $f$, $Y(f)$ is continuous almost surely.
\end{lemma}
\begin{proof}
Fix $f\in C^\infty$. By Aldous' tightness criterion (which we showed in section~\ref{sec:tightness}), we know that
\begin{equation*}
    \lim_{\delta\rightarrow 0}\limsup_{N\rightarrow\infty} \LL^N(w'_\delta(Y^N(f))\geq \epsilon) = 0.
\end{equation*}
Now note that
\begin{equation*}
    w_\delta(X) \leq 2w'_\delta(X) + \sup_t |X_t-X_{t-}|.
\end{equation*}
Since in our case the last term can be a.s. bounded by $2N^{-1/2}\|f\|_\infty$, we get
\begin{equation*}
    \lim_{\delta\rightarrow 0}\limsup_{N\rightarrow\infty} \LL^N(w_\delta(Y^N(f))\geq \epsilon) = 0.
\end{equation*}
This implies a.s. continuity of $Y(f)$.
\end{proof}

Now we can show that $Y$ satisfies a martingale problem under $\LL^*$.
\begin{prop}\label{mgaleprob}
Recall from~\eqref{eq:OU-mgale-def} that we define for $Y\in \D([0,T],(C^\infty)')$ and $f\in C^\infty$,
\begin{eqnarray}
    M_t^f&:=&Y_t(f)-Y_0(f) - \int_0^t Y_s(\Delta_Mf)\dd s\label{mgaleprobeq}\\
    N_t^f&:=& (M_t^f)^2 - 2t\rho(1-\rho)\int (\grad f)^2 \dd \overline V.\nonumber
\end{eqnarray}
Then for each $f\in C^\infty$, $M^f$ and $N^f$ are martingales with respect to the natural filtration under $\LL^*$.
\end{prop}
\begin{proof}
The proof is analogous to the proof of~\cite[Chapter 11 Prop 2.3]{kipnis1999scaling}. First recall
\begin{eqnarray*}
    M_t^{N,f}&=&Y_t(f)-Y_0(f) - \int_0^t Y_s(\el^Nf)\dd s\\
    N_t^{N,f}&=& (M_t^{N,f})^2 - \int_0^t \Gamma^{N,f}(s)\dd s.
\end{eqnarray*}
With a slight abuse of notation, we interpret $M_t^f$ (similarly $M_t^{N,f}, N_t^f, N_t^{N,f}$) here as a function $\D([0,T],(C^\infty)')\rightarrow \R$. When writing $\E_N M_t^f$ we mean the expectation of this function as a function of $Y^N$ and when writing $\E M_t^f$ we regard it as a function of the random element of $\D([0,T],(C^\infty)')$ with law $\LL^*$.

Fix $0\leq s\leq t \leq T$. We want to show that under $\LL^*$
\begin{equation*}
    \E [M_t^f| \mathscr{F}_s]= M_s^f\quad\text{and}\quad \E [N_t^f| \mathscr{F}_s]= N_s^f.
\end{equation*}
Fix $n\in \mathbb N, s\geq 0, 0\leq s_1 \leq..\leq s_n\leq s, H_1,..,H_n\in C^\infty, \Psi\in C_b(\R^n)$ and define
\begin{eqnarray*}
    I:&& \D([0,T],(C^\infty)')\rightarrow \R \\
    I:&& Y \mapsto \Psi(Y_{s_1}(H_1),..,Y_{s_n}(H_n)).
\end{eqnarray*}
Now it suffices to show that
\begin{equation}\label{eq:toshowuniqueness}
    \lim_{N\rightarrow\infty} \E_N M_t^{N,f}I(Y)  = \E M_t^fI(Y), \quad\quad \lim_{N\rightarrow\infty} \E_N N_t^{N,f}I(Y)  = \E N_t^fI(Y),
\end{equation}
since then by the martingale property of $M_t^{N,f}$
\begin{equation*}
    \E M_t^fI(Y) = \lim_{N\rightarrow\infty} \E_N M_t^{N,f}I(Y) = \lim_{N\rightarrow\infty} \E_N M_s^{N,f}I(Y) = \E M_s^fI(Y)
\end{equation*}
and analogous for the $N_t^f$ case, which implies that $M_t^f$ and $N_t^f$ are martingales under $\LL^*$.

We start with the first martingale.
First we show that we can replace $M_t^{N,f}$ by $M_t^f$ in the first expectation in~\eqref{eq:toshowuniqueness}. Using Jensen we see that
\begin{equation*}
    \left(\E_N M_t^{N,f}I(Y) - \E_N M_t^fI(Y)\right)^2 \leq
    \|\Psi\|_\infty^2\E_N (M_t^{N,f}-M_t^f)^2
    = \|\Psi\|_\infty^2\E_N \left( \int_0^t Y_s(\Delta_Mf)\dd s -\int_0^t Y_s(\el^Nf)\dd s \right)^2,
\end{equation*}
which goes to $0$ by lemma~\ref{lemma:replacelaplacian}.
Now it remains to show that
\begin{equation*}
    \lim_{N\rightarrow\infty} \E_N M_t^f I(Y)  = \E M_t^fI(Y).
\end{equation*}
First of all note that
\begin{eqnarray*}
    \E_N (M_t^f)^2 &\leq& 4 \left(\E_N Y_t(f)^2 + \E_N Y_0(f)^2+\E_N \left(\int_0^t Y_s(\Delta_Mf)\dd s\right)^2\right) \\
    &\leq& 2\rho(1-\rho) \int f^2\dd \overline V + t^2 \rho(1-\rho) \int (\Delta_Mf)^2\dd \overline V +o(1),
\end{eqnarray*}
which implies that there exists $C>0$ such that
\begin{equation*}
    \sup_{N\in\mathbb N} \E_N (M_t^fI(Y))^2\leq \|\Psi\|_\infty^2\sup_{N\in\mathbb N} \E_N (M_t^f)^2\leq C<\infty.
\end{equation*}
This implies that the $M_t^fI(Y)$ under $\LL^N$ are uniformly integrable, so it suffices to show that $M_t^fI(Y)$ under $\LL^N$ converges to $M_t^fI(Y)$ under $\LL^*$ in distribution. We proceed in steps. First of all, consider the mapping
\begin{eqnarray*}
    P_1:D([0,T],(C^\infty)') &\longrightarrow& D([0,T],\R)^{n+2}\\
    Y_\cdot &\longmapsto& (Y_\cdot(f),Y_\cdot(\Delta_M f),Y_\cdot(H_1),..,Y_\cdot(H_n))
\end{eqnarray*}
By~\cite[Thm 1.7]{jakubowski1986skorokhod} each of the components is continuous, hence $P_1$ is continuous. This implies that $P_1(Y^N)$ converges in distribution to $P_1(Y)$. Now consider the mapping
\begin{eqnarray*}
    P_2: D([0,T],\R)^{n+2}&\longrightarrow& \R\\
    (X^1,X^2,..,X^{n+2}) &\longmapsto& (X^1_t-X^1_0-\int_0^tX^2_s\dd s)\Psi(X^3_{s_1},..,X^{n+2}_{s_n}).
\end{eqnarray*}
Now suppose $(X^m){m\geq 1}$ is a sequence in $D([0,T],\R)^{n+2}$ (denoting $X^m=(X^{m,1},..,X^{m,n+2})$) that converges to $X\in D([0,T],\R)^{n+2}$ such that $X^i$ is a continuous path for each $i\leq n+2$. Note that this implies that for each $i\leq n$ $X^{m,i}$ converges uniformly to $X^i$. So in particular for fixed $i\leq n, t\in[0,T]$ $X^{m,i}_t$ converges to $X^i_t$ and, because of the uniform convergence, $\int_0^tX^{m,i}_s\dd s$ converges to $\int_0^t X^i_s\dd s$. Combining all of this with the knowledge that $\Psi$ is continuous, we obtain that $P_2(X^m)$ converges to $P_2(X)$. Since for each $f$, $Y(f)$ is continuous with probability $1$, we see that the $X\in D([0,T],\R)^{n+2}$ such that $X^i$ is a continuous path for each $i\leq n+2$ is a set of full measure under the measure on $X\in D([0,T],\R)^{n+2}$ induced by $\LL^*$ through $P_1$. Therefore the set of discontinuities of $P_2$ has measure $0$ under this measure. Hence by the Portmanteau theorem, we conclude that
\begin{equation*}
    M_tfI(Y^N) = P_2(P_1(Y^N))\rightarrow M_tfI(Y)
\end{equation*}
in distribution, which is what we wanted.
Note that the Portmanteau theorem is only valid in the context of metric spaces, this is the reason for the reduction to the (metric!) space $D([0,T],\R)^{n+2}$.


The proof for the second martingale is similar. We can estimate
\begin{eqnarray*}
    \E_N \left(N^{N,f}_t-N^f_t\right)^2&=&\E_N \left(\left((M_t^{N,f})^2 - \int_0^t \Gamma^{N,f}(s)\dd s \right) -\left((M_t^f)^2 - 2t\rho(1-\rho)\int (\grad f)^2 \dd \overline V\right)\right)^2\\
    &\leq& 2\E_N \left((M_t^{N,f})^2-(M_t^f)^2\right)^2 + 2\E_N \left(\int_0^t \Gamma^{N,f}(s)\dd s - 2t\rho(1-\rho)\int (\grad f)^2 \dd \overline V\right)^2
\end{eqnarray*}
The right term goes to $0$ by lemma~\ref{lemma:quadvarconv}. For the left term note that
\begin{eqnarray*}
    \E_N \left((M_t^{N,f})^2-(M_t^f)^2\right)^2 &=& \E_N (M_t^{N,f}-M_t^f)^2(M_t^{N,f}+M_t^f)^2 \leq \left(\E_N (M_t^{N,f}-M_t^f)^4\E_N(M_t^{N,f}+M_t^f)^4\right)^{1/2}\\
    &\leq& \left(8\E_N (M_t^{N,f}-M_t^f)^4\left(\E_N(M_t^{N,f})^4+\E_N(M_t^f)^4\right)\right)^{1/2}
\end{eqnarray*}
(where we used that $(a+b)^4\leq 8(a^4+b^4)$). Now we calculate
\begin{eqnarray}
    &&\E_N (M_t^{N,f}-M_t^f)^4\nonumber\\ 
     &=& \E_N \left(\int_0^t Y_s(\el^N f-\Delta_M f)\dd s\right)^4 \leq t^3 \int_0^t \E_N \left(Y_s(\el^N f-\Delta_M f)\right)^4 \dd s\label{eq:bound4mom1}\\
    &=& t^4 \frac{1}{N^2}\sum_{i,j,k,l=1}^N  (\el^N f-\Delta_M f)(p_i)..(\el^N f-\Delta_M f)(p_l)\E_N(\eta_0(p_i)-\rho)..(\eta_0(p_l)-\rho),
\end{eqnarray}
where we used that the expectation in the first line does not depend on $s$, so we can just set $s=0$. Now note that the expectation in the last line is only non-zero if every index is present either 2 or 4 times. This gives $O(N^2)$ non-zero terms. This means there is some constant $C>0$ such that
\begin{equation}\label{eq:bound4mom2}
    \E_N (M_t^{N,f}-M_t^f)^4 \leq t^4 C \sup_{i\leq N}|\el^N f-\Delta_M f)(p_i)|^4 = t^4 C E(N)^4,
\end{equation}
which goes to zero as $N$ goes to infinity.

Now we show that $\E_N(M_t^{N,f})^4$ is uniformly bounded, the term with $M_t^f$ can be treated analogously.
\begin{equation*}
    \E_N(M_t^{N,f})^4\leq 64 \left(\E_N Y_t(f)^4 + \E_N Y_0(f)^4 + \E_N \left(\int_0^t Y_s(\el^N f) \dd s\right)^4\right).
\end{equation*}
All these terms can be bounded as in~\eqref{eq:bound4mom1} to~\eqref{eq:bound4mom2}. One should further use that
\begin{equation*}
    \sup_{i\leq N} |\el^N f(p_i)|\leq \|\Delta_Mf\|_\infty + E(N) = O(1).
\end{equation*}
With the same kind of computations, one can show that $N^{N,f}_t\Psi(Y)$ under $\LL^N$ has uniformly bounded second moments, so is uniformly integrable. Than using continuity of the terms involved with respect to the uniform distance, we conclude in the same way as with $M^f_t$.
\end{proof}

Now we finally need to know that the martingale problem has unique solutions (given the initial conditions).
\begin{prop}
The martingale problem~\eqref{mgaleprobeq} together with the initial condition~\eqref{initdist} uniquely determine $\LL^*$ as a measure on $D([0,T],(C^\infty)')$.
\end{prop}
\begin{proof}
The proof of this theorem follows exactly like the proof of \cite[Chapter 11 Theorem 0.2]{kipnis1999scaling}, which is given in paragraph~4 of the same chapter. The idea of the proof is that for fixed $f\in C^\infty$ one can use the martingales $M^f$ and $N^f$ from the martingale problem to calculate the transition probabilities for the corresponding process $Y=(Y_t,t\geq 0)$. Since this process is Markov, the transition probabilities combined with the initial condition uniquely determine it.
\end{proof}

We conclude that every convergent subsequence $\LL_{N_k}$ converges to the same limit $\LL^*$. This implies that $\LL_N$ converges to $\LL^*$. Moreover, since $\LL^*$ satisfies the martingale problem~\eqref{eq:OU-mgale-def}, the limiting field is a generalized Ornstein-Uhlenbeck process.\\
\\

To conclude this section, we directly calculate the covariance structure of the limiting field. This is indeed the covariance given in~\eqref{eq:covOU} that one would expect from a generalized Ornstein-Uhlenbeck process.
\begin{prop}
For all $f,g\in C^\infty, t,s\geq 0$
\begin{equation*}
    \E[Y_{t+s}(f)Y_s(g)]=\rho(1-\rho) \int (S_{t}f)g \dd \overline V.
\end{equation*}
\end{prop}
\begin{proof}
We start to calculate the following covariance. To this we need to use duality of the SEP with a random walk, as is shown in~\cite[Thm 4.74]{liggett2010continuous}.
\begin{eqnarray*}
    &&\E \left[(\eta_t(p_i)-\rho)(\eta_0(p_j)-\rho)\right]
    = \int \E_{\eta'} \left[\eta_t(p_i)-\rho\right] (\eta'_0(p_j)-\rho) \nu_\rho(\dd\eta') \\
    &=& \int \E_{p_i} \left[(\eta'(X_t)-\rho)\right] (\eta'_0(p_j)-\rho) \nu_\rho(\dd\eta')
    =  \E_{p_i}  \left[\int(\eta'(X_t)-\rho)(\eta'(p_j)-\rho)\nu_\rho(\dd\eta')\right] \\
    &=& \E_{p_i} \1_{X_t=p_j}\rho(1-\rho) = \rho(1-\rho)\p_{p_i}(X_t=p_j),
\end{eqnarray*}
where under $\E_{p_i}$, $X$ is the random walk starting from $p_i$.
Now we calculate the covariances for $\LL^N$.
\begin{eqnarray*}
    \E[Y^N_t(f)Y^N_0(g)] &=& \frac{1}{N} \sum_{i,j=1}^N f(p_i)g(p_j) \E \left[(\eta_t(p_i)-\rho)(\eta_0(p_j)-\rho)\right]\\
    &=&\frac{\rho(1-\rho)}{N} \sum_{i=1}^N f(p_i) \sum_{j=1}^N  \p_{p_i}(X_t=p_j) g(p_j)\\
    &=&\frac{\rho(1-\rho)}{N} \sum_{i=1}^N f(p_i) S^N_tg(p_i),
\end{eqnarray*}
where $(S^N_t,t\geq 0)$ is the semigroup corresponding to the random walk $X^N$ on the grid points. By remark 3.4 from~\cite{vanGinkel/Redig:2018}, for each $f\in C^\infty$,
\begin{equation*}
    \lim_{N\rightarrow\infty}\sup_{1\leq i\leq N} \left| S_t^Nf|_{G_N}(p_i)-S_tf(p_i)\right| = 0,
\end{equation*}
where $(S_t,t\geq 0)$ is the semigroup of Brownian motion. Using this and stationarity of $\eta$, we obtain
\begin{eqnarray*}
    \E[Y_{t+s}(f)Y_s(g)] &=& \lim_{N\rightarrow\infty} \E[Y^N_{t+s}(f)Y^N_s(g)] = \lim_{N\rightarrow\infty} \E[Y^N_t(f)Y^N_0(g)] \\
    &=& \lim_{N\rightarrow\infty} \frac{\rho(1-\rho)}{N} \sum_{i=1}^N f(p_i) S^N_tg(p_i) = \rho(1-\rho) \int fS_tg \dd \overline V.
\end{eqnarray*}
\end{proof} 